\documentclass[12pt]{amsart}

\usepackage[left=2.5cm,right=2.5cm,top=2.8cm,bottom=2.8cm]{geometry}
\usepackage{amsmath,amssymb,amsfonts,amsthm}
\usepackage{mathtools}
\usepackage{graphicx}
\usepackage[backref,colorlinks=true,linkcolor=blue,citecolor=blue,urlcolor=blue]{hyperref}

\newtheorem{theorem}{Theorem}[section]
\newtheorem{lemma}[theorem]{Lemma}
\newtheorem{corollary}[theorem]{Corollary}

\newtheorem{conjecture}[theorem]{Conjecture}
\theoremstyle{definition}
\newtheorem{definition}[theorem]{Definition}
\theoremstyle{remark}
\newtheorem{remark}[theorem]{Remark}

\numberwithin{equation}{section}
\numberwithin{figure}{section}

\DeclareMathOperator{\Mat}{Mat}

\newcommand{\Q}{\mathbb Q}

\newcommand{\bfe}{\mathbf e}

\newcommand{\T}{\mathsf T}

\title[Smith normal forms for coalescences]{Smith normal forms for coalescences at cospectral vertices}

\author[Y.-Z. Fan]{Yi-Zheng Fan*}
\address{Center for Pure Mathematics, School of Mathematical Sciences, Anhui University, Hefei 230601, P. R. China}
\email{fanyz@ahu.edu.cn}
\thanks{*Supported by National Natural Science Foundation of China (Grant No. 12331012).}

\author[K. Zhang]{Kuo Zhang}
\address{School of Mathematical Sciences, Anhui University, Hefei 230601, P. R. China}
\email{zhangk@stu.ahu.edu.cn}

\author[W. Wang]{Wei Wang$^\sharp$}
\address{School of Mathematics and Statistics, Xi'an Jiaotong University, Xi'an 710049, P. R. China}
\email{wang\_weiw@xjtu.edu.cn}
\thanks{$^\sharp$Corresponding author. Supported by National Natural Science Foundation of China (Grant No. 12371357).}


\begin{document}

\begin{abstract}
Let $L_\mu(G)=A(G)-\mu D(G)$ be the generalized $\mu$-adjacency matrix of a finite graph $G$.
Fan, Xing, Zhang, and Wang constructed pairs of non-degree-similar trees for which the Smith normal forms of the matrices $tI-L_\mu(G)$ over $\Q(\mu)[t]$ coincide, and conjectured that their construction remains valid when the attached rooted path is replaced by an arbitrary rooted tree.
We prove this conjecture as a consequence of a more general coalescence theorem: if a finite graph $H$ has two vertices $u$ and $v$ that are cospectral for $L_\mu(H)$, then, for every finite rooted graph $R$ with root $r$, the matrices
\[
        tI-L_\mu(R(r)\odot H(u))
        \quad\text{and}\quad
        tI-L_\mu(R(r)\odot H(v))
\]
have the same Smith normal form over $\Q(\mu)[t]$, where $\odot$ denotes coalescence of rooted graphs.
The proof uses an orthogonal intertwiner over a real closed extension field.
\end{abstract}

\maketitle

\section{Introduction}
Let $G$ be a finite simple graph with adjacency matrix $A(G)$ and degree matrix $D(G)$.
Following Godsil and Sun \cite{GodsilSun}, two graphs $G_1$ and $G_2$ are called \emph{degree-similar} if there exists an invertible matrix $M$ such that
\[ M^{-1}A(G_1)M=A(G_2), \quad  M^{-1}D(G_1)M=D(G_2). \]
This condition implies similarity for several standard graph matrices, including the adjacency matrix, the Laplacian, the signless Laplacian and, when there are no isolated vertices, the normalized Laplacian.

The generalized $\mu$-adjacency matrix of $G$ is
\[  L_\mu(G):=A(G)-\mu D(G),\]
and the corresponding $\mu$-polynomial is
\[   \psi(G,t,\mu):=\det(tI-L_\mu(G)).\]
Wang et al. \cite{WangLiLuXu} asked whether two graphs with the same $\mu$-polynomial must be degree-similar via an orthogonal matrix.
Godsil and Sun \cite{GodsilSun} gave a negative answer by constructing pairs of graphs that have the same $\mu$-polynomial but are not degree-similar.

If $G_1$ and $G_2$ are degree-similar, then $L_\mu(G_1)$ and $L_\mu(G_2)$ are similar over $\Q(\mu)$.
Godsil and Sun proved that, over a field $F$, two matrices $B_1$ and $B_2$ are similar if and only if the matrices $tI-B_1$ and $tI-B_2$ have the same Smith normal form over $F[t]$ \cite[Theorem 9.1]{GodsilSun}.
Consequently, degree-similar graphs have the same Smith normal form for $tI-L_\mu(G)$ over $\Q(\mu)[t]$.
Godsil and Sun asked whether the converse holds \cite[Problem 10.2]{GodsilSun}.

Fan, Xing, Zhang and Wang gave a negative answer by constructing pairs of trees that are not degree-similar but whose matrices $tI-L_\mu(G)$ have the same Smith normal form over $\Q(\mu)[t]$ \cite{FanXingZhangWang}.
Their construction is based on McKay's rooted tree pair \cite{McKay}.
Let $\T$ be the tree shown in Figure~\ref{fig:H}, which was introduced by McKay in his construction of non-isomorphic cospectral trees.
In \cite[Lemma 11]{FanXingZhangWang}, the authors proved that if a rooted path is attached to $\T$ at vertex $4$ and at vertex $7$, then the resulting two trees have the same Smith normal form over $\Q(\mu)[t]$.
They then proposed the following conjecture.

\begin{conjecture}[{\cite{FanXingZhangWang}}]
\label{conj:FXZW}
Let $\T$ be the tree in Figure~\ref{fig:H}, and let $R$ be any nontrivial rooted tree with root $r$.
Define
\[ G_4:=R(r)\odot \T(4),\quad   G_7:=R(r)\odot \T(7),\]
where $\odot$ denotes coalescence of rooted graphs.
Then
\[  tI-L_\mu(G_4)  \quad\text{and}\quad  tI-L_\mu(G_7)\]
have the same Smith normal form over $\Q(\mu)[t]$.
\end{conjecture}

The purpose of this paper is to prove Conjecture~\ref{conj:FXZW}.
We prove a more general statement.
Let $H$ be any finite graph, and let $u,v\in V(H)$ be cospectral for $L_\mu(H)$.
Then attaching an arbitrary finite rooted graph $R$ at $u$ or at $v$ gives two coalescences whose generalized $\mu$-adjacency matrices have the same Smith normal form over $\Q(\mu)[t]$.
Conjecture~\ref{conj:FXZW} follows by taking $H=\T$ and $(u,v)=(4,7)$.
The main linear-algebraic ingredient is an orthogonal involution, over a real closed extension field of $\Q(\mu)$, that commutes with $L_\mu(H)$ and sends $\bfe_u$ to $\bfe_v$.

\vspace{1mm}

\begin{figure}[htbp]
  \centering
  \includegraphics[scale=0.8]{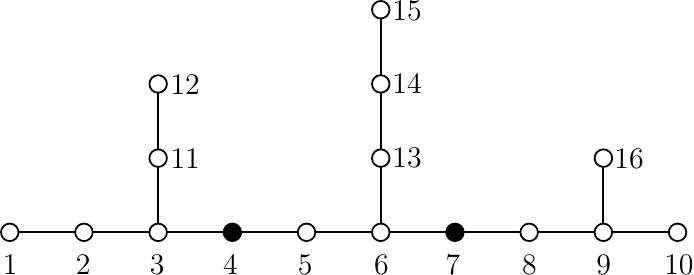}
  \caption{\small The tree $\T$ on $16$ vertices; the distinguished vertices are $4$ and $7$.}
  \label{fig:H}
\end{figure}

\section{Main result}

For a matrix $M$ with rows and columns indexed by a finite set and for an index $u$, let $M(u)$ denote the principal submatrix obtained from $M$ by deleting the row and column indexed by $u$.
For a graph $G$ and a vertex $u$, let $G(u)$ denote a rooted copy of $G$ with root $u$.
If $R$ is a rooted graph with root $r$, then the \emph{coalescence} $R(r)\odot G(u)$ denotes the graph obtained by identifying $r$ with the vertex $u$ of $G$.

\begin{definition}
Let $M$ be a square matrix over a field.
Two indices $u$ and $v$ are called \emph{cospectral for $M$} if
\[   \det(tI-M(u))=\det(tI-M(v)).\]
\end{definition}

The idea of cospectral vertices goes back to Schwenk \cite{Schwenk}, who used it to construct non-isomorphic cospectral trees.
The characterization of cospectral vertices for adjacency matrix of a graph was established in \cite{GodsilMcKay,GodsilSmith} over the real field.

We shall use real closed fields only as a technical device.
Recall that a field \(F\) is \emph{real closed} if it is orderable, every positive element of \(F\) is a square, and every polynomial over \(F\) of odd degree has a root in \(F\);
equivalently, \(F(\sqrt{-1})\) is algebraically closed; see, for example,
\cite[Chapter XIII]{LangAlgebra}.

We fix an ordering of \(K=\mathbb Q(\mu)\), for instance the ordering obtained
by identifying \(K\) with \(\mathbb Q(t)\) and declaring \(t\) to be a positive
infinite transcendental element.
Let \(F\) be a real closure of \(K\) with respect to this ordering.
The particular choice of ordering plays no role in what follows.
We only use the fact that symmetric matrices over the real closed field
\(F\) admit an orthogonal spectral decomposition.
After the required similarity is obtained over \(F\), it is descended back to
\(K=\mathbb Q(\mu)\) by using rational canonical form.

For a graph \(G\) and a vertex \(v\in V(G)\), or more generally for a square
matrix \(M\) and an index \(v\), let \(\bfe_v\) denote the standard basis
column vector indexed by \(V(G)\), or by the index set of \(M\), respectively.
Thus \(\bfe_v\) has its only nonzero entry, equal to \(1\), in position \(v\).

\begin{lemma}[Orthogonal intertwiner]\label{lem:cospectral-intertwiner}
Let $F$ be a real closed field, and let $M\in\Mat_n(F)$ be a symmetric matrix.
Let $u$ and $v$ be two indices.
Suppose that $u$ and $v$ are cospectral for $M$.
Then there exists an orthogonal involution $Q\in\Mat_n(F)$ such that
\[   QM=MQ, \quad Q\bfe_u=\bfe_v, \quad Q\bfe_v=\bfe_u.\]
\end{lemma}

\begin{proof}
Since $F$ is real closed and $M$ is symmetric, the spectral theorem over real closed fields gives an orthogonal spectral decomposition
\[        M=\sum_{\theta\in\Theta}\theta E_\theta,\]
where $\Theta$ is the set of distinct eigenvalues of $M$, and $E_\theta$ is the orthogonal projection onto the $\theta$-eigenspace.

By Cramer's rule,
\[ \bfe_u^\top(tI-M)^{-1}\bfe_u   =  \frac{\det(tI-M(u))}{\det(tI-M)},\]
and similarly
\[ \bfe_v^\top(tI-M)^{-1}\bfe_v   =  \frac{\det(tI-M(v))}{\det(tI-M)}.\]
The cospectrality assumption gives
\[ \bfe_u^\top(tI-M)^{-1}\bfe_u   =  \bfe_v^\top(tI-M)^{-1}\bfe_v.\]
On the other hand,
\[ (tI-M)^{-1}  =   \sum_{\theta\in\Theta}\frac{1}{t-\theta}E_\theta.\]
Therefore
\[ \sum_{\theta\in\Theta}\frac{\bfe_u^\top E_\theta\bfe_u}{t-\theta}
        = \sum_{\theta\in\Theta} \frac{\bfe_v^\top E_\theta\bfe_v}{t-\theta}.\]
By uniqueness of partial fractions,
\[  \bfe_u^\top E_\theta\bfe_u    =    \bfe_v^\top E_\theta\bfe_v \]
for every $\theta\in\Theta$.
Equivalently,
\[    \|E_\theta\bfe_u\|^2=\|E_\theta\bfe_v\|^2.\]

Fix $\theta\in\Theta$ and put
\[   x_\theta:=E_\theta\bfe_u,   \quad     y_\theta:=E_\theta\bfe_v.\]
Then $x_\theta$ and $y_\theta$ lie in the $\theta$-eigenspace and have the same norm.
We claim that there is an orthogonal involution $Q_\theta$ of the $\theta$-eigenspace such that $Q_\theta x_\theta=y_\theta$.
If $x_\theta=y_\theta$, take $Q_\theta$ to be the identity.
Otherwise, define $Q_\theta$ to be the reflection in the hyperplane perpendicular to $x_\theta-y_\theta$:
\[   Q_\theta z
=   z -2\frac{\langle z,x_\theta-y_\theta\rangle}{\langle x_\theta-y_\theta,x_\theta-y_\theta\rangle}
        (x_\theta-y_\theta).
\]
The denominator is nonzero, since \(x_\theta-y_\theta\neq 0\) and the standard
quadratic form is anisotropic over a real closed field.
Since $\|x_\theta\|=\|y_\theta\|$, this reflection sends $x_\theta$ to $y_\theta$.
It is orthogonal and satisfies $Q_\theta^2=I$.

Taking the orthogonal direct sum of the maps $Q_\theta$ over all eigenspaces, we obtain an orthogonal involution $Q$ on $F^n$.
Since $Q$ preserves each eigenspace of $M$, it commutes with $M$.
Moreover,
\[ Q\bfe_u  = Q\left(\sum_{\theta\in\Theta}E_\theta\bfe_u\right)
            = \sum_{\theta\in\Theta}Q_\theta E_\theta\bfe_u
            = \sum_{\theta\in\Theta}E_\theta\bfe_v
            = \bfe_v.\]
Since $Q^2=I$, it follows that $Q\bfe_v=\bfe_u$.
This proves the lemma.
\end{proof}

\begin{remark}
The preceding lemma is related to the notion of strongly cospectral vertices,
which was systematically studied by Godsil and Smith \cite{GodsilSmith} in
connection with continuous quantum walks.
Recall that two indices $u$ and $v$ are \emph{strongly cospectral} for a real symmetric matrix $M$ with spectral decomposition $M=\sum_\theta \theta E_\theta$
if
\[        E_\theta\bfe_u=\pm E_\theta\bfe_v\]
for every eigenvalue $\theta$.
In this case one may choose
\[  Q=\sum_\theta \varepsilon_\theta E_\theta,  \quad   \varepsilon_\theta\in\{\pm1\},\]
which is an orthogonal involution satisfying
\[   QM=MQ,   \quad  Q\bfe_u=\bfe_v.\]
Thus strongly cospectral vertices give a special case in which the intertwiner
has the particularly simple form above, whereas Lemma~\ref{lem:cospectral-intertwiner}
requires only ordinary cospectrality.
\end{remark}

We use the following standard criterion relating similarity and Smith normal forms; see \cite[Theorem 9.1]{GodsilSun}, \cite[Theorem 7.6.1]{LanTis}, or \cite[Theorem 2.1.4]{Fri}.

\begin{lemma}\label{lem:Sim-SNF}
Two matrices $B_1$ and $B_2$ over a field $F$ are similar if and only if the matrices $tI-B_1$ and $tI-B_2$ have the same Smith normal form over $F[t]$.
\end{lemma}

We now state and prove the general coalescence result.
No acyclicity assumption is needed on either the core graph or the attached rooted graph.

\begin{theorem}\label{thm:main}
Let $H$ be a finite simple graph, and let $u,v\in V(H)$ be distinct vertices that are cospectral for $L_\mu(H)$.
Let $R$ be any finite rooted graph with root $r$.
Define
\[   G_u:=R(r)\odot H(u),   \quad    G_v:=R(r)\odot H(v).\]
Then
\[   tI-L_\mu(G_u)   \quad\text{and}\quad  tI-L_\mu(G_v)\]
have the same Smith normal form over $\Q(\mu)[t]$.
\end{theorem}

\begin{proof}
Let
\[     K:=\Q(\mu),    \quad     B:=L_\mu(H)=A(H)-\mu D(H).\]
By assumption,
\begin{equation}\label{eq:general-root-deleted-equality}
        \det(tI-B(u))=\det(tI-B(v)).
\end{equation}
Let \(F\) be the real closure of \(K\) fixed above.
The equality \eqref{eq:general-root-deleted-equality} remains valid over $F[t]$.
Since $B$ is symmetric, Lemma~\ref{lem:cospectral-intertwiner} yields an orthogonal involution $Q\in\Mat_{V(H)}(F)$ such that
\begin{equation}\label{eq:general-Q-properties}
        QB=BQ,    \quad    Q\bfe_u=\bfe_v,   \quad   Q\bfe_v=\bfe_u.
\end{equation}

Put $R^\circ:=R-r$.
Let $k=\deg_R(r)$, and let $c$ be the column vector indexed by $V(R^\circ)$ defined by
\[        c_v=
        \begin{cases}
        1, & v \text{ is adjacent to } r \text{ in } R,\\
        0, & \text{otherwise}.
        \end{cases}
\]
Let
\[    C:=A(R^\circ)-\mu D_R[V(R^\circ)], \]
where $D_R[V(R^\circ)]$ denotes the principal submatrix of the degree matrix of $R$ indexed by $V(R^\circ)$.
Thus the degrees in this diagonal block are the degrees in $R$, not the degrees in $R^\circ$.

Order the vertices of $G_u$ and $G_v$ so that the vertices of $H$ come first, followed by the vertices of $R^\circ$.
Then
\begin{equation}\label{eq:block-Gu}
        L_\mu(G_u)
        = \begin{pmatrix}
        B-\mu k\bfe_u\bfe_u^\top & \bfe_uc^\top\\
        c\bfe_u^\top              & C
        \end{pmatrix},
\end{equation}
and
\begin{equation}\label{eq:block-Gv}
        L_\mu(G_v)
        = \begin{pmatrix}
        B-\mu k\bfe_v\bfe_v^\top & \bfe_vc^\top\\
        c\bfe_v^\top              & C
        \end{pmatrix}.
\end{equation}
Define
\[   \widehat Q:=
        \begin{pmatrix}
        Q&0\\
        0&I
        \end{pmatrix}.
\]
Since $Q$ is an orthogonal involution, $Q^{-1}=Q^\top=Q$.
Using \eqref{eq:general-Q-properties}, \eqref{eq:block-Gu}, and \eqref{eq:block-Gv}, we get
\[  \widehat Q^{-1}L_\mu(G_u)\widehat Q=L_\mu(G_v).\]
Thus $L_\mu(G_u)$ and $L_\mu(G_v)$ are similar over $F$.

Since both matrices have entries in $K$, similarity over the extension field $F$ implies similarity over $K$ by rational canonical form.
Indeed, two matrices over $K$ that become similar over an extension field have the same invariant factors over $K[t]$, hence the same rational canonical form over $K$.
Therefore $L_\mu(G_u)$ and $L_\mu(G_v)$ are similar over $K$.
By Lemma~\ref{lem:Sim-SNF},
\[      tI-L_\mu(G_u)
        \quad\text{and}\quad
        tI-L_\mu(G_v)
\]
have the same Smith normal form over $K[t]=\Q(\mu)[t]$.
\end{proof}

\begin{corollary}[Conjecture~\ref{conj:FXZW}]\label{cor:conj15}
Let $\T$ be the tree in Figure~\ref{fig:H}, with vertices labelled as in that figure.
Let $R$ be any finite rooted graph with root $r$.
Define
\[
        G_4:=R(r)\odot \T(4),
        \quad
        G_7:=R(r)\odot \T(7).
\]
Then
\[
        tI-L_\mu(G_4)
        \quad\text{and}\quad
        tI-L_\mu(G_7)
\]
have the same Smith normal form over $\Q(\mu)[t]$.
In particular, Conjecture~\ref{conj:FXZW} holds when $R$ is any nontrivial rooted tree.
\end{corollary}

\begin{proof}
Let $B:=L_\mu(\T)$.
By the computation in the proof of \cite[Lemma 3.1]{GodsilSun}, the two distinguished vertices $4$ and $7$ of $\T$ are cospectral for $B$, that is,
\[    \det(tI-B(4))=\det(tI-B(7)). \]
Therefore the assertion follows from Theorem~\ref{thm:main}, applied with $H=\T$, $u=4$, and $v=7$.
\end{proof}

\begin{corollary}
Let $R$ be any nontrivial rooted tree with root $r$, and let
\[
        G_4:=R(r)\odot \T(4),
        \quad
        G_7:=R(r)\odot \T(7).
\]
Then $G_4$ and $G_7$ have the same Smith normal form for $tI-L_\mu$ over $\Q(\mu)[t]$, but they are not degree-similar.
\end{corollary}

\begin{proof}
The equality of Smith normal forms follows from Corollary~\ref{cor:conj15}.
By McKay's theorem, as used in \cite{FanXingZhangWang}, the two trees $G_4$ and $G_7$ are non-isomorphic whenever $R$ is a nontrivial rooted tree.
Also, two trees are degree-similar if and only if they are isomorphic \cite[Theorem 3.2]{GodsilSun}.
Therefore $G_4$ and $G_7$ are not degree-similar.
\end{proof}

\begin{remark}
Lemma 11 of \cite{FanXingZhangWang} proves the path case by showing that the $(n-1)$-st determinant divisor equals $1$, and hence the Smith normal form has the special form
\[   1,\ldots,1,\psi(t,\mu).\]
For a general rooted graph $R$, and a fortiori in the general setting of Theorem~\ref{thm:main}, this special form need not hold.
The argument above proves the stronger and more natural conclusion needed for Conjecture~\ref{conj:FXZW}: the two complete Smith normal forms are equal.
\end{remark}

\end{document}